\documentclass[12pt]{article}
\usepackage{hyperref}
\usepackage{amsmath}
\usepackage{amssymb}
\usepackage{amsthm}
\usepackage{framed}
\usepackage{multicol}
\usepackage[utf8]{inputenc}
\usepackage{graphicx}
\usepackage[all]{xy}
\usepackage{color}
\usepackage{mathrsfs}
\usepackage{mathtools}
\usepackage{authblk}
\usepackage{tikz-cd}
\usepackage{MnSymbol}
\usepackage{kbordermatrix}

\numberwithin{equation}{section}

\newcommand{\Gr}{\mathrm{Gr}}

\newcommand{\C}{\mathbb{C}}

\newcommand{\ol}{\overline}

\newcommand{\Addresses}{{
  \bigskip
  \footnotesize

  K.~Timchenko, \textsc{Department of Mathematics, University of Notre Dame, Notre Dame, IN 46556}\par\nopagebreak
  \textit{E-mail address:} \texttt{ktimchen@nd.edu}

}}

\date{}

\begin{document}

\title{Characteristic cycles for K-orbits on Grassmannians}
\author{Kostiantyn Timchenko}

\maketitle

\begin{abstract}
We compute characteristic cycles of the IC-sheaves associated to the $K$-orbits on Grassmannians.
\end{abstract}

\newtheorem{theorem}{Theorem}
\newtheorem{lemma}[theorem]{Lemma}
\newtheorem{proposition}[theorem]{Proposition}
\newtheorem{corollary}[theorem]{Corollary}
\newtheorem{conjecture}[theorem]{Conjecture}

\theoremstyle{definition}
\newtheorem{definition}[theorem]{Definition}
\newtheorem{question}[theorem]{Question}
\newtheorem{comment}[theorem]{Comment}
\newtheorem{remark}[theorem]{Remark}
\newtheorem{example}[theorem]{Example}
\newtheorem*{thm1}{Theorem 1}

\section{Introduction}
The purpose of this paper is to establish new results about characteristic cycles for orbit closures of interest in geometric representation theory. 
Let $G$ be a complex group acting on a smooth projective variety $X$ with a finite number of orbits. Given a $G$-orbit $Q$, we denote by $\mathcal{L}_Q$ the IC-extension of the trivial local system on $Q$ and by $T_{Q}^*X$ the conormal bundle to the orbit $Q$. We write the characteristic cycle of $\mathcal{L}_Q$ as follows
\begin{equation}
CC(\mathcal{L}_Q) = \overline{T_{Q}^*X} + \sum _{Q\neq Q'\subset \overline{Q}} m_{Q',Q} \, \overline{T_{Q'}^*X}
\end{equation}
for some non-negative numbers $m_{Q',Q}$. We say that $CC(\mathcal{L}_Q)$ is irreducible if $CC(\mathcal{L}_Q)=\overline{T_{Q}^*X}$, i.e., all the $m_{Q',Q}=0$. The determination of these numbers is a difficult problem.

One instance of this problem is when $X$ is a complete flag variety of a complex reductive group $G$ and $Q$ is a Schubert cell, in which case $X=G/B$ and $Q$ is a $B$-orbit.
 In a 1980 paper, Kazhdan and Lusztig \cite{KL} conjectured that in type A, the characteristic cycle of the IC-sheaf $\mathcal{L}_Q$ is irreducible. In 1997, Kashiwara and Saito \cite{KSa} provided a counterexample to this conjecture for $G=GL(8)$.  More recently, Williamson \cite{W} showed that  $m_{Q',Q}$ could be non-zero even when $Q'$ and $Q$ correspond to the elements of the Weyl group that lie in the same two-sided cell.

A special case of this conjecture was proved by Bressler-Finkelberg-Lunts \cite{BFL}, who showed that the statement is true when the flag manifold is replaced by the Grassmannian. Using different methods, this was later generalized by Boe-Fu \cite{BF} to the case of Schubert varieties in the classical homogeneous Hermitian spaces.

In this paper, we switch our attention to $K$-orbits on Grassmannians. More precisely, let $\theta$ be an involution on the group $GL(n)$ and let $K$ be the  identity component of the fixed set of $\theta$. In this case, $K$ belongs to one of the three families: $GL(p)\times GL(q)$, $Sp(n)$ and $SO(n)$. $Sp(n)$ exists only for $n$ even. 
Each $K$ acts on $Gr(k,n)$ with a finite number of orbits. The $GL(p)\times GL(q)$-orbits $Q(s,t)$ are parametrized by the pairs of numbers $(s,t)$, where $s\leq p$, $t\leq q$ and $s+t\leq k$ (see \ref{definition of the pq orbits}). The $Sp(n)$-orbits are denoted by $Q(i)$, where $i=k(\text{mod} \, 2)$ and $0\leq i \leq k$. The $SO(n)$-obits are denoted by $Q(i)$, for $i=0,\dots,k$ (see \ref{definition of Qi orbits}). We prove the following: 

\begin{theorem}\label{thm1}
 Let $X$ be the Grassmannian $Gr(k,n)$. The characteristic cycles of the IC-sheaves associated to the $K$-orbits are:
 \begin{enumerate}
 \item[a)] irreducible for $K=GL(p)\times GL(q) $, $p+q=n$, and $ Sp(n)$,
 \item[b)] not irreducible in general for $K=SO(n)$. Consider the $SO(n)$-orbits $Q(0), Q(1),\dots, Q(k)$ on $Gr(k,n)$. Then $CC(\mathcal{L}_{Q(i)})$ is irreducible for $i$ even or $i=k$. When $i$ is odd, $CC(\mathcal{L}_{Q(i)})$ has two components (with the exception of $CC(\mathcal{L}_{Q(k-1)})$ in $ Gr(k,2k)$, $k$ is even,  that has three connected components).
 \end{enumerate}
\end{theorem}
 
 This paper is organized as follows. In the following section we discuss the geometry of $GL(p)\times GL(q)$-orbits in the Grassmannian. In particular, we describe the small resolutions of closures of these orbits.  We then review the notion of a microlocal fiber and prove that the characteristic cycles associated to the $GL(p)\times GL(q)$-orbits are irreducible (Theorem \ref{pq theorem}). Our strategy follows the lines of Bressler-Finkelberg-Lunts argument: we show that the  microlocal fiber of the resolution is empty (outside of a dense orbit).

 In Section \ref{section so and sp} we discuss the orbits of $SO(n)$ and $Sp(n)$ on the Grassmannian $Gr(k,n)$. We use the language of degeneracy loci to connect their geometry to the geometry of symmetric and skew-symmetric matrix rank stratifications. The characteristic cycles are known in the latter case, this allows us to translate the statements about characteristic cycles from the matrix side to the $K$-orbit side (Theorem \ref{so sp orbits theorem}).

\section{Orbits for the General Linear Groups} \label{section pq orbits}
We start with the case of a pair $(GL(n), GL(p)\times GL(q))$ for $p+q=n$. We follow the notation from \cite{BE}. Fix a decomposition of $\C^n=\C^p\oplus \C^q$. Here, $GL(p)$ acts on the first summand, $\C^p$, and $GL(q)$ acts on the second, $\C^q$. We write $Gr(i,n)$ for the Grassmannian of $i$-planes in $\C^p$. For the rest of the section we assume that $p\geq q$ and $n-k\geq k$. 

$GL(n)$ acts on the Grassmannian $Gr(k,n)$. This gives an action of $K$ on $Gr(k,n)$. The $K$ - orbits are the subsets \begin{equation}\label{definition of the pq orbits}
Q(s,t) = \{ U\in Gr(k,n) | \dim U\cap \C^p = s, \dim U\cap C^q = t \},
\end{equation}
where $s+t\leq k$, $s\leq p$ and $t\leq q$. Note that $Q(s',t')\subset \ol{Q(s,t)}$ if and only if $s'\geq s$ and $t'\geq t$.

In \cite{BE}, Barbasch and Evens introduce two resolutions of singularities for $\ol{Q(s',t')}$. First, consider
\begin{equation}
Z(s,t) = \{ (U,V,W)\in Gr(k,n)\times Gr(s,p) \times Gr(t,q) |  V\subset \C^p\cap U, W\subset \C^q\cap U \}.
\end{equation}
Barbasch and Evens show that the projection on the first factor $\theta: Z\to \ol{Q(s,t)} $ is a resolution of singularities. For $n-k\geq p$ this resolution is small.

To deal with the $n-k<p$ case another incidence correspondence is used. Consider
\begin{equation}
\tilde{Z}(s,t) = \{ (U,V,W)\in Gr(k,n)\times Gr(k+p-s,n)\times Gr(t,q) |  U+{\C^p} \subset V, W\subset U\cap \C^q \}.
\end{equation}

As before, the projection $\tilde{\theta}:\tilde{Z} \to \ol{Q(s,t)} $ is a resolution of singularities. Assuming $n-k\leq p$, this resolution is small.

\subsection{Characteristic Cycles}\label{sec char cycles}
In \cite{BFL} , Bressler, Finkelberg and Lunts use Zelevinsky's resolutions for Schubert varieties to prove the irreducibility of the characteristic cycles. We briefly review their strategy.
As in the introduction, we can write the characteristic cycle associated to an orbit as follows
\begin{equation}
CC(\mathcal{L}_Q) = \overline{T_{Q}^*X}+ \Sigma _{Q'\subset \overline{Q}} m_{Q',Q} \overline{T_{Q'}^*X}.
\end{equation}
Following Brylinski \cite{Br}  the multiplicity $m_{Q',Q}$ can be computed in terms of vanishing cycles. Take a general point $(x,\xi)\in T^*_{Q'}X$ and a function $f$ on $X$, vanishing along $Q'$ and satisfying $d_x f=\xi$. Let $\Phi_f $ denote the functor of vanishing cycles  with respect to $f$. Then the multiplicity
\begin{equation}
m_{Q'Q} = \chi_x(\Phi_f \mathcal{L}_{Q}),
\end{equation}
where $\chi_x$ is the stalk Euler characteristic at $x$.
 
Suppose we are in the situation where $Q$ admits a (proper and) small resolution of singularities $\theta: Z\to \ol{Q}$. Now 
\begin{equation}
m_{Q'Q} = \chi_x(\Phi_f \mathcal{L}_{Q}) = \chi_x(\Phi_f R\theta _* \C) = \chi_x(R\theta _* \Phi_{f \circ \theta}\C).
\end{equation}
The second equality follows from the map $\theta$ being small and the third by the proper base change.

Now recall the definition of a microlocal fiber. For $z\in Z$, the codifferential of $\theta$ at $z$ is $d_z^{*}\theta:T^*_{\theta(z)} \ol{Q} \to T^*_{z} Z$. 
\begin{definition}
The microlocal fiber of $\theta$ over $( x,\xi)$ is the set of points $z\in \theta^{-1}(x)$ such that $d_z^{*}\theta (\xi) = 0$.
\end{definition}
The support of the complex $\Phi_{f \circ \theta}\C|_{\theta^{-1}(x)}$ is contained in the microlocal fiber over $(x,\xi)$[KS]. 
\begin{remark}\label{remark}
If the microlocal fiber of $\theta$ over a general point $(x,\xi)$ is empty, then the multiplicity $m_{Q'Q}$ is zero.
\end{remark}

We return to the setup of (\ref{section pq orbits}). Consider a $K$-orbit $Q(s,t)$. Let $\theta:Z\to Gr(k,n)$ be one of the two small resolutions of $\ol{Q(s,t)}$.  We are going to prove the following 
\begin{proposition}\label{prop cone is empty}
 For any proper $K$-orbit $Q'\subset \ol{Q(s,t)}$ and any sufficiently general $(x,\xi)$ in $T_{Q'}^*Gr(k,n)$ the microlocal fiber of $\theta$ over $(x,\xi)$ is empty (equivalently the cone $\cup _{z\in \theta^{-1}(x)} \ker d_z^{*}\theta$ is a proper subvariety of $T_{Q'}^*Gr(k,n)_x$).
\end{proposition}

We compute the codifferentials first. As usual, we identify the tangent space to a Grassmannian $T_U Gr(k,n)$ with the vector space $Hom(U,\C^n/U)$. Similarly, we identify $T^*_U Gr(k,n) = Hom(\C^n/U,U)$.
\begin{lemma}\label{codifferential}
\begin{itemize}
\item
For $\theta : Z(s,t)\to Gr(k,n)$ and a  triple $(U,V,W)$ in $Z_{(s,t)}$, then $\ker d_{(U,V,W)}^*\theta$ is 
\begin{equation}
\{h+l\in  Hom(\C^n/U,U) \, | \, h(\C^n) \subset V, h|_{\C^p} =0, l(\C^n) \subset W, l|_{\C^q} =0 \}.
\end{equation}
\item
For $\tilde{\theta}: \tilde{Z}(s,t)\to Gr(k,n)$ and a  triple $(U,V,W)$ in $\tilde{Z}_{(s,t)}$, then  $\ker d_{(U,V,W)}^*\tilde{\theta}$ is 
\begin{equation}
\{h+l\in  Hom(\C^n/U,U) \, | \, h|_V = 0, h (\C^n)\subset \C^q , l(\C^n) \subset W, l|_{\C^q} =0 \}.
\end{equation}
\end{itemize}
\end{lemma}

\begin{proof}
a) Consider the case of $Z(s,t)$ first. We are interested in the codifferential of the map $\theta=\theta_{(s,t)}: Z(s,t)\to Gr(k,n) $. Note that we have $Z(s,t)=Z(s,0)\times_{Gr(k,n)} Z(0,t)$ and  $\theta$ fits into the Cartesian diagram

\[
  \begin{tikzcd}
    Z_{(s,t)} \arrow{r}{} \arrow{d}{} & Z_{(s,0)} \arrow{d}{\theta_{(s,0)}} \\
    Z_{(0,t)} \arrow{r}{\theta_{(0,t)}} & Gr(k,n)
  \end{tikzcd}
\]
where the maps $Z_{(s,t)}\to Z_{(s,0)}$ and $Z_{(s,t)} \to Z_{(0,t)}$ forget $W$ and $V$ respectively. 
Fix a triple $(U,V,W)$ in $Z_{(s,t)}$. On the level of cotangent spaces this diagram becomes 
\[
  \begin{tikzcd}
    T^*_{(U,V,W)}Z_{(s,t)}  & T^*_{(U,V)}Z_{(s,0)}\arrow{l}{f} \\
   T^*_{(U,W)} Z_{(0,t)} \arrow{u}{} & T^*_{U}Gr(k,n) \arrow{l}{d^*_{(U,W)}\theta_{(0,t)}} \arrow{u}{d^*_{(U,V)}\theta_{(s,0)} }
  \end{tikzcd}
\]
Now $d^*_{(U,V,W)}\theta = f\circ d^*_{(U,V)}\theta_{(s,0)} $ and $\ker (f\circ d^*_{(U,V)}\theta_{(s,0)}) = \ker d^*_{(U,V)}\theta_{(s,0)} +\ker d^*_{(U,W)}\theta_{(0,t)}$ because the square is co-Cartesian.

Note that $Z_{(s,0)}$ is the incidence correspondence $\{ (U,V)| V\subset U  \} $ in $Gr(k,n)\times Gr(s,p)$ and the map $\theta_{(s,0)}$ is a projection $(U,V)\mapsto U$. 
 Now by [BFL, Lemma 3.1 a)] ("rightmost peak" case)
\begin{equation}
\ker d_{(U,V)}^*\theta_{(s,0)} = \{ h\in Hom(\C^n/U,U) \, | \, h(\C^n) \subset V, h|_{\C^p} =0 \}.
\end{equation}
In the same manner 
\begin{equation}
\ker d_{(U,W)}^*\theta_{(0,t)} = \{ l\in Hom(\C^n/U,U) \, | \, l(\C^n) \subset W, l|_{\C^q} =0 \}.
\end{equation}
Altogether, $\ker d_{(U,V,W)}^*\theta$ is 
\begin{equation}
\{h+l\in  Hom(\C^n/U,U) \, | \, h(\C^n) \subset V, h|_{\C^p} =0, l(\C^n) \subset W, l|_{\C^q} =0 \}.
\end{equation}

b)Consider $\tilde{\theta}: \tilde{Z}(s,t)\to Gr(k,n)$. Note that $\tilde{Z}_{(s,t)} = \tilde{Z}_{(s,0)}\times_{Gr(k,n)}\tilde{Z}_{(0,t)}$. By definition, $\tilde{Z}_{(0,t)} = Z_{(0,t)}$. 

Fix a triple $(U,V,W)$ in $\tilde{Z}_{(s,t)}$.
Similarly to part a), $d^*_{(U,V,W)}\tilde{\theta} = \ker d^*_{(U,V)}\tilde{\theta}_{(s,0)} +\ker d^*_{(U,W)}\tilde{\theta}_{(0,t)}$. Note that $ \tilde{Z}_{(s,0)} $ is a set of the form $\{ (U,V) \, | \, U+\C^p \subset V \}$ in $Gr(k,n)\times Gr(k+p-s,n)$ 
 Now \cite{BFL}, Lemma 3.1 a), ("leftmost peak" case) describes the codifferential of the map $\tilde{\theta}: \tilde{Z}_{(s,0)} \to Gr(k,n)$ as follows 
\begin{equation}
\ker d_{(U,V)}^*\theta = \{ h\in Hom(\C^n/U,U) \, | \, h|_V = 0, h (\C^n)\subset \C^q  \}.
\end{equation} 
We conclude that 
$\ker d_{(U,V,W)}^*\tilde{\theta}$ is 
\begin{equation}
\{h+l\in  Hom(\C^n/U,U) \, | \, h|_V = 0, h (\C^n)\subset \C^q , l(\C^n) \subset W, l|_{\C^q} =0 \}.
\end{equation}\end{proof}

Let $U$ be a point in $Gr(k,n)$. Then $GL(p)\times Gl(q)\cdot U$ equals $Q(s,t)$ for some $s$ and $t$. Then $T_U Q(s,t)$ is equal to the subspace of $Hom(U,\C^n/U)$ that sends $U\cap \C^p$ to $\C^p/U\cap \C^p$ and $U\cap \C^q$ to $\C^q/U\cap \C^q$. Under the natural pairing between $Hom(U,\C^n/U)$ and $Hom(\C^n/U, U)$, $T^*_{Q(s,t)}Gr(k,n)_U$ is the subspace of $Hom(\C^n/U, U)$ that sends $\C^p/U$ to $U\cap \C^q$, $\C^q/U$ to $U\cap \C^p$ and $\C^p/U \cap \C^q/U$ to $0$.

We can represent a conormal vector in $T^*_{Q(s,t)}Gr(k,n)_U$ by a $k$ by $n-k$ matrix of the shape 

\begin{equation}
\kbordermatrix{&\C^p/U  & \C^p/U \cap \C^q/U & \C^q/U\\
U\cap \C^p &0 & 0 & * \\
U\cap \C^q & * & 0 & 0  \\
 & 0 & 0 &   0
}.
\end{equation}
We return to the proof of Proposition \ref{prop cone is empty}. 
\begin{proof}
a) We start with the case of $n-k\geq p$ and a resolution $\theta : Z(s,t) \to \ol{Q(s,t)}$. Fix a properly included orbit $Q(s',t')\subset Q(s,t)$ together with a point $U$ in $Q(s',t')$. Then $s'\geq s$ and $t'\geq t$ and one of the inequalities is strict. According to Lemma \ref{codifferential}, the subvariety $\cup_{\theta^{-1}(U)} \ker d^*_{(U,V,W)}\theta$ is a union of the subspaces 
\begin{equation}\label{rkrk}
\{h+l\in  Hom(\C^n/U,U) \, | \, h(\C^n) \subset V, h|_{\C^p} =0, l(\C^n) \subset W, l|_{\C^q} =0 \}
\end{equation}
for all pairs $(V,W)\in Gr(s,p)\times Gr(t,q)$ such that $V\subset U\cap \C^p, W\subset U\cap \C^q$. 
For an element $(U,m)$ in $T^*_{Q(s,t)}Gr(k,n)_U$ the maximal rank of the matrix $m$ is equal to $r=min(s',n-k-p+s')+min(t',n-k-q+t')$. We claim that the matrices in (\ref{rkrk}) are of the strictly smaller rank. Indeed, the condition $n-k\geq p$ implies $n-k-p+s'\geq s'$. Also $n-k-q+t'> t'$ since $p\geq q$ and $2k<n$. Thus $r=s'+t'$. The image conditions $h(\C^n) \subset V $ and $l(\C^n) \subset W$ decrease the maximal possible rank to $s+t$.

b) Consider the case of $\tilde{\theta} : \tilde{Z}(s,t) \to \ol{Q(s,t)}$ when $n-k<p$. As before, we take $U\in Q(s',t')\subset Q(s,t)$ where the inclusion is strict. By Lemma 2.1\ref{codifferential}, $\cup_{\tilde{\theta}^{-1}(U)} \ker d^*_{(U,V,W)}\tilde{\theta}$  is a union of the following subspaces 
\begin{equation}
\{h+l\in  Hom(\C^n/U,U) \, | \, h|_V = 0, h (\C^n)\subset \C^q , l(\C^n) \subset W, l|_{\C^q} =0 \}
\end{equation}
for all $(V,W)\in Gr(k,n)\times Gr(k+p-s,n)$ with the conditions $ U\oplus{\C^p} \subset V, W\subset U\cap \C^q $. As before, the maximal rank of a matrix in  $T^*_{Q(s,t)}Gr(k,n)_U$ is $r=min(s',n-k-p+s')+min(t',n-k-q+t')= min(s',n-k-p+s')+t'$. We have $n-k<p$ thus $r=n-k-p+s' + t'$. Notice that $\dim \C^n/U - \dim V/U = n-k-(k+p-s-k) = n-k-p+s$. This way, $\dim \C^n/U - \dim V/U <  n-k-p+s'.$ Thus both kernel $V\subset \ker h$ and image $l(\C^n) \subset W$ conditions decrease the maximal rank from $r$ to $n-k -p +s + t$.
\end{proof}

Proposition \ref{prop cone is empty} together with Remark \ref{remark} imply the first part of Theorem \ref{thm1}:
\begin{theorem} \label{pq theorem}
The characteristic cycles of IC-sheaves associated to  $GL(p)\times GL(q)$-orbits on $Gr(k,n)$ are irreducible.
\end{theorem}

\section{Orbits for the Symplectic and Orthogonal Groups} \label{section so and sp}

\subsection{Rank stratifications for symmetric and skew-symmetric matrices}\label{section rank stratifications}
We start with a reminder on matrix rank stratifications. Let $X = Hom(\C^n,\C^n)$ be the space of $n\times n $ matrices. There is a non-degenerate symmetric bilinear form $\langle \cdot, \cdot\rangle $ on $X$  given by $\langle A, B \rangle = tr(AB^{T})$. Note that the space of symmetric matrices is orthogonal to the space of skew-symmetric matrices for this form, so the form is non-degenerate on both symmetric and skew-symmetric matrices.  

Let $M_{symm}\subset Hom(\C^n,\C^n)$ be a subspace of symmetric matrices. Consider the action of $GL(n)$ on $X$ defined by $g\cdot M = gMg^T$, where $g\in GL(n)$ and $M\in X$. There are $n+1$ orbits $O_0, O_1, \dots , O_n$. Each orbit $O_i$ consists of symmetric matrices of rank $i$. 

Similarly, define $M_{skew}$ to be the subspace of $n\times n$ skew-symmetric matrices. Let $m=\left \lfloor n/2\right \rfloor$. The $GL(k)$-action produces $m+1$ orbits $O_0, O_1, \dots, O_{m}$. Now each orbit $O_{i}$ consists of skew-symmetric matrices of rank $2i$.

The conormal space to an orbit can be easily described as follows. Let $x\in Hom(\C^n,\C^n)$ be symmetric (or skew-symmetric). Denote by $O$ the corresponding $GL(n)$-orbit through $x$.  Take $C$ to be a  symmetric (resp. skew-symmetric) matrix in $T_x Hom(\C^n,\C^n)$. Then $C$ is orthogonal to $T_x O$ if and only if $tr(C(Yx+xY^T)) = 0$ for any $Y$ in $  Hom(\C^n,\C^n)$. Using $tr(CYx) = tr(CxY^T)$ for $x,C$ symmetric (resp. skew-symmetric), we see this is equivalent to $x C= 0 $.

From Raicu \cite{Rai} and Lorincz-Walther \cite{LW} (see also Braden-Grinberg \cite{BG}), we know characteristic cycles for the matrix rank stratifications.
 As before, we denote by $\mathcal{L}_O$ the IC-extension of the trivial local system on the orbit $O$. 
\begin{theorem}\cite{Rai}\label{rank stratification}
\begin{itemize}
\item Let $O_0, O_1, \dots, O_{m}$ be the orbits in $M_{skew}$. Then each $CC(\mathcal{L}_{O_i})$ is irreducible, so that $CC(\mathcal{L}_{O_i}) = [\overline{T_{O_i}^*X}]$.
\item Let $O_0, O_1, \dots , O_n$ be the orbits in $M_{symm}$. Then  
$CC(\mathcal{L}_{O_i}) = [\overline{T_{O_i}^*X}]$ for $n-i$ even or  $i=0$, and $CC(\mathcal{L}_{O_i}) = [\overline{T_{O_i}^*X}] + [\overline{T_{O_{i-1}}^*X}]$  for $n-i$ odd.
\end{itemize}
\end{theorem}

\subsection{Non-characteristic maps}
We recall the definition of a non-characteristic map here. Consider a map $f: X\to Y$ between smooth algebraic varieties. Let $\omega_f:X\times_Y T^*Y \to T^*Y$ and $\rho_f: X\times _Y T^*Y \to T^*X$ be the canonical morphisms. 
\begin{definition}
Let $V$ be a closed conic subset of $T^*Y$. We say that $f: X\to Y$ is non-characteristic for $V$ if $\ker \rho_f \cap \omega^{-1}_f(V)$ is in the zero section of $X\times_Y T^*Y$. If $\mathcal{G}^\bullet$ is an object in $D^b(Y)$ and $f$ is non-characteristic for $Ch(\mathcal{G}^\bullet)$ we say that $f$ is non-characteristic for $\mathcal{G}^\bullet$.
\end{definition}
Note that if $f$ is smooth, then $f$ is non-characteristic for any closed conic subset. If $f:X\to Y$ is a closed embedding and $V\subset T^*Y$ is a closed conic subset then $f$ is non-characteristic for $V$ if and only if 
\begin{equation} \label{noncharact closed embedding}
(V)_x\cap (T^*_X Y)_x = \{0\}
\end{equation}
for all $x\in X$.
We will use the following 
\begin{theorem}\cite{HK} \label{HottaK}
Let $f:X\to Y$ be a map between smooth algebraic varieties. Consider $\mathcal{G}^\bullet$ a perverse sheaf on $Y$. If $f$ is non-characteristic for $\mathcal{G}^\bullet$ then 
\begin{equation}
CC(f^{-1}\mathcal{G}^\bullet) = \rho_f \omega_f^{-1}(CC(\mathcal{G}^\bullet)).
\end{equation}
\end{theorem}

\subsection{Characteristic Cycles and Degeneracy Loci}

Let $E\to X$ be a vector bundle over a smooth variety $X$ with a fiber $V$. 
Suppose that $E = E_{\leq n}\supset E_{\leq n-1} \supset \dots \supset E_{\leq 0}$ is a stratification by closed subsets with smooth strata $E_i=E_{\leq i}- E_{\leq i-1}$, $i=0,\dots, n$.  We assume that this restricts to a stratification of $V = V_{\leq n}\supset V_{\leq n-1} \supset \dots \supset V_{\leq 0}$. Let  $s:X\to E$ be a section of $E$. Then the subvariety $X_{\leq k} = s^{-1}(E_{\leq k})$ in $X$ is called a degeneracy locus. For example, if $E=Hom(F,G)$, where $F$ and $G$ are two vector bundles on $X$, then the stratification on $E$ can be taken to be the matrix rank stratification in each $Hom(F,G)_x$ for $x$ in $X$. In this case, $X_{\leq k}$ is exactly the set $\{x\in X \, | \, \text{rank } s_x \leq k\}$.

\begin{lemma}\label{transversality}
Suppose $s$ is transversal to each of the sets $E_{i}$, $i=0,\dots, n$. Then $s^{-1}$ preserves characteristic cycles, i.e. if $CC(\mathcal{L}_{V_{k}}) = \Sigma _{	i\leq k} m_{i,k} \, \overline{T_{V_{i}}^*V} $, then $CC(\mathcal{L}_{X_k}) = \Sigma _{	i\leq k} m_{i,k} \, \overline{T_{X_{i}}^*X}$ with the same coefficients $m_{i,k}$.
\end{lemma}
\begin{proof}
Notice that $s$ being transversal to each $E_i$ implies $X_{i} = s^{-1}(E_{i})$ is smooth and the expression $\mathcal{L}_{X_{i}}$ makes sense. In a neighborhood $U$ of $x\in X$, $s$ takes values in $V$ and is transversal to the stratification of $V$ by $V_i$'s. Thus $s$ is a normally nonsingular map from $U$ to $V$ in the sense of \cite{GM} and $s^{-1}\mathcal{L}_{V_i} = \mathcal{L}_{X_i}[\dim V - \dim X]$. The same transversality implies that $s$ is a non-characteristic map for any constructible complex on $V$. In particular, $s$ is non-characteristic for each of the $\overline{T_{V_{i}}^*V}$ using Equation (\ref{noncharact closed embedding})	. Notice that $\rho_s \omega_s ^{-1}(\overline{T_{V_{i}}^*V}) = \ol{T_{X_{i}}^*X}$ and we can use Theorem \ref{HottaK} to conclude the proof. 
\end{proof}

In simple terms, Lemma \ref{transversality} can be stated as follows. Suppose the section $s$ is general enough. Then the structure of singularities for degeneracy loci is the same one as for the fiber $V$.

\subsection{Orbits for the Symplectic and Orthogonal Groups}
Let $B$ be a nondegenerate bilinear form on $\C^n$. We assume $B$ is symmetric or skew-symmetric. Let $K$ denote the isometry group of $B$. The $K$-orbits on $Gr(k,n)$ can be described as follows 
\begin{equation}\label{definition of Qi orbits}
Q(i) = \{ U\in Gr(k,n) | \dim(rad(U)) = i \}.
\end{equation}
If $B$ is skew-symmetric, then $K=Sp(n)$ and the sets $Q(i)$, where $0\leq i \leq k$ and $i=k(\text{mod} \,2) $, are precisely the $Sp(n)$-orbits. When $B$ is symmetric, then $K = O(n)$ and there are $k+1$ orbits $Q(0), Q(1),\dots, Q(k)$. However, all $O(n)$-orbits $Q(i)$ are connected, except for $Q(n/2)$ ($n$ is even), which consists of two closed $SO(n)$-orbits (\cite{ACGH}, p. 102). In what follows, we use the same notation $Q(0), Q(1),\dots, Q(k)$ for the set of $SO(n)$-orbits, remembering that the $O(n)$-orbit $Q(n/2)$ actually consists of two closed $SO(n)$-orbits for an even $n$.

Note that $Q(i)\subset\ol{Q(j)}$ if and only if $i\geq j$. One has a resolution of singularities \cite{BE} analogous to (\ref{section pq orbits}). Consider 
\begin{equation}
Z_i = \{ (U,V)\in Gr(k,n)\times Gr(i,n) |  V\subset rad(B|_U) \}
\end{equation}
with $\theta :Z_i \to Q(i)$ a projection to the first factor. However, these resolutions are not small in general and we cannot employ the strategy from Section \ref{sec char cycles}. Instead we relate $K$-orbits to matrix rank stratifications from Section  \ref{section rank stratifications}. 

In what follows, we describe the $K$-orbits as degeneracy loci. 
Let $\mathcal{S}$ be the tautological bundle over a Grassmannian $Gr(k,n)$. Form the bundle $Hom(\mathcal{S},\mathcal{S}^*)$ over  $Gr(k,n)$. The choice of a bilinear form $B$ corresponds to a section $s: Gr(k,n) \to Hom(\mathcal{S},\mathcal{S}^*)$ via $U\mapsto B|_{U}:U\to U^*$. 
Now $rk B|_U = k - rad(U)$ and each $Q(i)$ is a preimage of rank $(k - i)$ matrices in $Hom(\mathcal{S},\mathcal{S}^*)$.

Given $U \subset V$, we identify $Hom(U,U^*)$ with bilinear forms on $U$ in the usual manner.   Thus, it makes sense to speak of symmetric (skew-symmetric) matrices in $Hom(U,U^*)$. We denote by $M_{U,symm}$  ($M_{U,skew}$) the subspace of symmetric (skew-symmetric) matrices in $Hom(U,U^*)$. When $U$ is clear from the context, we omit the subscript $U$ from $M_{U,symm}$ and write simply $M_{symm}$. 

Let $U$ be the span $\langle e_1,\dots,e_k \rangle$ in $Gr(k,n)$ with $k\geq n-k$, identify $\C^n/U$ with $\langle e_{k+1},\dots,e_n \rangle$. This gives an identification 
\begin{multline*}
T_U(\Gr(k,n)) = Hom(U,\C^n/U) = Hom(U, \langle e_{k+1}, ..., e_n\rangle) \\
=\{v+A(v) : v\in U, A \in Hom(U, \langle e_{k+1}, ..., e_n \rangle) \},
\end{multline*}
 which is an affine open neighborhood of $U$ in $\Gr(k,n)$. 

\begin{proposition}\label{prop map sB}  
 Choose $B$ to be symmetric (or skew-symmetric). In the neighborhood of $U$, the map $s:Hom(U,\C^n/U)\to Hom(U,U^*)$ is transversal to the rank stratification of symmetric (resp. skew-symmetric) matrices.
\end{proposition}
\begin{proof}
 
We start with the case of $SO(n)$ acting on $Gr(k,n)$. Here $n$ can be even or odd. Take $J$ to be the $n\times n$ matrix 
$\begin{pmatrix} 
 &  & 1\\
 & \udots & \\
1 &  & 
\end{pmatrix}$.  This gives us the non-degenerate symmetric form $B(v,w) = v^T J w$ on $\C^{n}$. Take a $k$-plane $U = \langle e_1,\dots,e_k \rangle$. Then the neighborhood  $Hom(U,\C^n/U)$ of $U$ meets a closed $SO(n)$-orbit $Q(k)$ through the point $U$. In the neighborhood $Hom(U,\C^n/U)$, the  section $s :Hom(U,\C^n/U) \to M_{U,symm}$ is 
\begin{equation}
A\mapsto \left( (v,w)\mapsto (v+Av)^T J (w+Aw)  \right),
\end{equation}
where $A\in Hom(\langle e_1,\dots,e_k \rangle,\langle e_{k+1},\dots, e_n\rangle )$.
For $v=(v_1, \dots , v_k)$ and $w=(w_1,\dots,w_k)$ in $U$,
\begin{equation}
s(A)(v,w) = \sum_{i=1}^{n-k} v_i (Aw)_{n-i+1} + \sum_{i=n-k+1}^{k} v_i w_{n-i+1} +  \sum_{i=k+1}^{n} (Av)_i w_{n-i+1}.
\end{equation}
As a symmetric $k\times k$ matrix, $s(A)$ has the form
\begin{equation} \label{shape of sbA} 
\begin{pmatrix} U & Y \\ Z & J \end{pmatrix},
\end{equation} 
where $U$ is a $n-k$ by $n-k$ matrix, 
$Y$ is $n-k$ by $2k-n$, $Z$ is $2k-n$ by $n-k$, and $J$ is the $2k-n$ by $2k-n$ matrix with $1$'s
along the antidiagonal, as above.

We claim that $s$ is transversal to any orbit $O$ in $M_{symm}$. Consider $x\in S\cap O$ and $C\in T_x M_{symm}$. Suppose C is perpendicular both to the tangent space $T_x im(s)$ and to $T_x O$. Now  $C\perp T_x im(s)$ implies that $C$ is of the form 
\begin{equation}\label{shape of C}
\begin{pmatrix} 0 & 0 \\ 0 & W \end{pmatrix},
\end{equation}
where the blocks have the same shape as in Equation (\ref{shape of sbA}).
 Recall that by Section \ref{section rank stratifications} $C$ being perpendicular to $T_x O$ is equivalent to $xC=0$. Equations (\ref{shape of C}) and (\ref{shape of sbA}) together with $xC=0$ imply $C=0$.

 The case of $Sp(n)$, $n$ is even, is analogous. Take $J$ to be the $n\times n$ matrix  
\begin{equation}
\begin{pmatrix} 
 &  &  &  &  & 1\\
 &  &  &  & \udots & \\
 &  &  & 1 &  & \\
 &  & -1 &  &  & \\
 & \udots &  &  &  & \\
-1 &  &  &  &  & \\
\end{pmatrix}.
\end{equation}
This gives us the non-degenerate skew-symmetric form $B(v,w) = v^T J w$ on $\C^{n}$.
Then $s(A)$ is a $k$ by $k$ skew-symmetric matrix of the following shape: 
\begin{equation} 
\begin{pmatrix} U & Y \\ Z & J \end{pmatrix}.
\end{equation} 
Here $U$ is a $n-k$ by $n-k$ matrix, 
$Y$ is $n-k$ by $2k-n$, $Z$ is $2k-n$ by $n-k$, and $J$ is the $2k-n$ by $2k-n$ matrix filled with $1$'s and $-1$'s
along the antidiagonal, as above. The rest of the argument goes through without changes.
\end{proof}

\begin{remark}
Recall that when $n$ is even and $k=n/2$ the set $Q(k)$ is a union of two closed $SO(n)$-orbits in $Gr(k,n)$. One of them contains the point $U=\langle e_1,\dots,e_n \rangle$ and another one - the point $V=\langle e_{n+1},\dots,e_{2n}\rangle$ in $Gr(k,n)$. A similar computation shows that Proposition \ref{prop map sB} is true if one replaces $U$ with $V$ in $Gr(k,n)$. 
\end{remark}

\begin{theorem} \label{so sp orbits theorem}
a) Consider the $Sp(n)$-orbits  $Q(i)$ on $Gr(k,n)$. Then, the characteristic cycle of  $\mathcal{L}_{Q(i)}$ is irreducible.

\noindent b) Consider the $SO(n)$-orbits $Q(0), Q(1),\dots, Q(k)$ on $Gr(k,n)$. Then, the characteristic cycle of  $\mathcal{L}_{Q(i)}$ is irreducible for $i$ even or $i=k$, and for $i$ odd $$CC(\mathcal{L}_{Q(i)})= \ol{T^*_{Q(i)} Gr(k,n)} + \ol{T^*_{Q(i+1)} Gr(k,n)}.$$  When $n$ is even and $k=n/2$, the set $Q(k)$ is a union of two closed $SO(n)$-orbits. Hence the characteristic cycle of $\mathcal{L}_{Q(k-1)}$ has three irreducible components when $k$ is even.
\end{theorem}

\begin{proof}
First assume that $k\geq n-k$ as in Proposition  \ref{prop map sB}. Recall that $s:Hom(U,\C^n/U) \to Hom(U,U^*)$ takes $U$ to $B|_{U}$ and $Q(i)=s^{-1} (O(k-i))$ in both cases. We start with the case a). By Proposition  \ref{prop map sB}, $s$ is transversal to the rank stratification of symmetric $k$ by $k$ matrices. Now Lemma \ref{transversality} allows us to pull back characteristic cycles from Theorem \ref{rank stratification} and we are done. The proof of part b) is analogous to part a). 
To deal with the case when $k \leq n-k$ notice that $Gr(k, \C^n)$ is $K$-equivariantly isomorphic to $Gr(n-k, \C^n)$, using the $K$-isomorphisms $Gr(k, \C^{n}) \to Gr(n-k, (\C^{n})^*)$, and $\C^{n}\to (\C^{n})^*$ given by the form $B$.
\end{proof}

\Addresses
\end{document}